\documentclass[10pt,a4paper]{article}
\usepackage[latin2]{inputenc}
\usepackage[english]{babel}

\makeatletter
\let\@fnsymbol\@arabic
\makeatother

\usepackage[english]{babel}

\usepackage{amsmath,amsfonts,amstext,amsbsy,amsopn,amssymb,amsthm,amscd}
\usepackage{amsxtra,latexsym}
\usepackage{exscale}

\newtheorem{theorem}{Theorem}[section]

\newtheorem{definition}[theorem]{Definition}

\newtheorem{corollary}[theorem]{Corollary}

\numberwithin{equation}{section}

\def\neweq#1{\begin{equation}\label{#1}}
\def\endeq{\end{equation}}
\def\eq#1{(\ref{#1})}

\newcommand{\R}{\mathbb{R}}
\newcommand{\N}{\mathbb{N}}

\newcommand{\eps}{\varepsilon}

\usepackage{graphicx,mathrsfs}
\usepackage{psfrag}

\usepackage{color}
\def\trait #1 #2 #3 {\vrule width #1pt height #2pt depth #3pt}
\def\fin{
    \trait .3 5 0
    \trait 5 .3 0
    \kern-5pt
    \trait 5 5 -4.7
    \trait 0.3 5 0
\medskip}

\begin{document}
\title{Which residual mode captures the energy of the dominating mode in second order Hamiltonian systems?}

\author{ Elvise Berchio\thanks{
Department of Mathematical Sciences,
Politecnico di Torino,
Corso Duca degli Abruzzi 24, 10129 Torino, Italy,
 E-Mail: \texttt{elvise.berchio@polito.it}}
\and
Filippo Gazzola\thanks{ Department of Mathematics,
Politecnico di Milano,
Piazza Leonardo da Vinci 32,
20133 Milano, Italy, E-Mail: \texttt{filippo.gazzola@polimi.it}}
\and
Chiara Zanini\thanks{
Department of Mathematical Sciences,
Politecnico di Torino,
Corso Duca degli Abruzzi 24, 10129 Torino, Italy, 
E-Mail: \texttt{chiara.zanini@polito.it}}
}

\date{ }

\maketitle
\begin{abstract}
Motivated by the instability of suspension bridges, we consider a class of second order Hamiltonian systems where one component initially holds almost
all the energy of the system. We show that if the total energy is sufficiently small then it remains on this component, whereas if the total energy is
larger it may transfer to the other components. Through Mathieu equations we explain the precise mechanism which governs the energy transfer.
\end{abstract}
\bigskip

\medskip


\bigskip
\section{Introduction and motivation}

The spectacular collapse of the Tacoma Narrows Bridge (occurred on November 7, 1940, see \cite{ammann,scott}) raised many questions
on the instability of suspension bridges. Soon after the collapse, several theoretical and experimental studies have been performed
\cite{bleichsolo,bleich,pittel,rocard}. The main issue was to understand the origin of the instability \cite{stein1} and, in particular, how could
vertical oscillations be suddenly transformed into destructive torsional oscillations. The focus was essentially on the aerodynamic instability
\cite{scantom} but no conclusive answer was found: in the last few years, the problem of aerodynamic instability of suspension bridges is still under
study \cite{como}. Only very recently, the attention has turned to the nonlinear behavior of structures \cite{lacarbonara}.\par In \cite{argaz} and \cite{bergaz}  the structural instability of suspension bridges has been highlighted by analyzing two fairly different isolated models. In \cite{argaz} the bridge was seen as a number of interacting parallel rods representing the cross sections of the bridge, each one having two degrees of freedom: the vertical displacement
of the barycenter and the torsional angle. A torsional instability was numerically found: if vertical displacements are sufficiently large then small torsional angles may suddenly grow up leading to the collapse of the bridge. The main tools to reach this result were suitable Poincar\'e
maps \cite{meccel}. See also \cite{mck} for some aerodynamics effects. In \cite{bergaz} the bridge was modeled as a degenerate plate, named fish-bone by the authors, where the midline of the plate
was seen as a beam and virtual orthogonal cross sections were considered free to rotate about their center placed on the beam. See also \cite{bergaz2} where the aerodynamic forces were introduced in the model. The same torsional
instability was found, both numerically and theoretically, and the instability was justified through the analysis of suitable Hill
equations \cite{hill}.\par
It is clear that there is a relation between these two models and approaches. This is probably due to the connection between Poincar\'e
and Hill, as testified in \cite{poincarehill}; their work takes the origin from celestial mechanics and, as we just saw, it applies as
well to suspension bridges. The results in \cite{argaz,bergaz} lead to the very same conclusion: if vertical oscillations are small enough then small
initial torsional oscillations remain small for all the time, whereas if vertical oscillations are large then small torsional oscillations can suddenly
become wider. This gives an answer to a long-standing question raised by the Tacoma Narrows Bridge collapse, see \cite{ammann,scott}, namely how can destructive torsional oscillations suddenly appear in a vertically oscillating bridge. The main core in both \cite{argaz,bergaz} is the
stability analysis of vertical modes, that is, how can a bridge oscillating as an almost pure vertical mode suddenly transfer part of the energy to a
torsional mode. We investigate this phenomenon by considering a class of second order Hamiltonian systems such as
\neweq{hamiltonian}
\ddot{y_{i}}+\lambda_i^2y_i+U_{y_i}(Y)=0\, ,\qquad Y=(y_1,\ldots,y_n)\in\mathbb{R}^n
\endeq
for some $n\ge2$, $\lambda_i>0$ and some potential $U\in C^1(\R^n,\R),$ where $U_{y_i}$ denotes the partial derivative of $U$ with respect to $y_i$. The heart of the matter is to study the evolution of the solutions of
\eq{hamiltonian} satisfying the initial conditions
\neweq{bigsmall}
y_1(0)=\zeta_0\, ,\ \dot{y}_1(0)=\zeta_1\, ,\qquad \sum_{i=2}^n(|y_i(0)|+|\dot{y}_i(0)|)\ll|\zeta_0|+|\zeta_1|
\endeq
for $\zeta_0,\zeta_1\in\R$; due to these uneven boundary conditions, we call $y_1$ the {\it dominating mode} and $y_i$ (for $i=2,...,n$) the
{\it residual modes}. And the main question is to establish if the unique solution $Y=Y(t)$ of \eq{hamiltonian}-\eq{bigsmall} has small residual modes
for every time $t>0$. It was shown in \cite{argaz,bergaz} that this is true provided that $|\zeta_0|+|\zeta_1|$ is sufficiently small whereas it may become false
if $|\zeta_0|+|\zeta_1|$ is sufficiently large. The typical pictures describing the instability of \eq{hamiltonian}-\eq{bigsmall} are as in
Figure \ref{typical}.

\begin{figure}[ht]
\begin{center}
{\includegraphics[height=33mm, width=60mm]{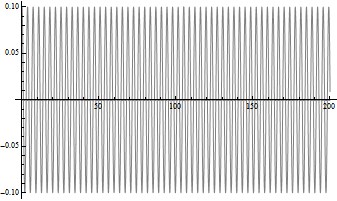}}\qquad{\includegraphics[height=33mm, width=60mm]{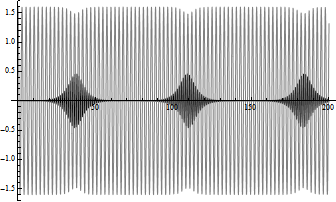}}
\caption{Stable and unstable oscillations.}\label{typical}
\end{center}
\end{figure}
\par
In both pictures, the gray oscillations represent the dominating mode whereas the black oscillations represent the largest component among the residual
modes. In the left picture, the initial data $\zeta_0$ and $\zeta_1$ in \eq{bigsmall} are small and the only large component of $Y$ is $y_1$
for all $t>0$, no black oscillations are visible. In the right picture, the initial data $\zeta_0$ and $\zeta_1$ in \eq{bigsmall} are larger and one
may see a large oscillation also in one of the residual modes: this mode suddenly grows up by capturing some energy from $y_1$ which
decreases its amplitude of oscillation when the transfer of energy occurs. This is what we call instability of $y_1$ (for large energies) and it can
be seen for many different forms of the potential $U$ in \eq{hamiltonian}, see \cite{argaz}.\par
A question was left open in \cite{argaz,bergaz}: which residual mode first captures the energy of $y_1$? Moreover, which is the criterion
governing the transfer of energy? The relevance of these questions relies on the possibility to understand which kind of oscillating mode will
first appear in the bridge when enough energy is inside the structure. In particular, this could help to prevent the appearance of the destructive
torsional oscillations. The main purpose of this paper is to give a sound answer to these questions.\par
We consider a simple prototype problem like \eq{hamiltonian}-\eq{bigsmall} with $n=3$. We choose a potential $U$ in such a way that the linearized problem
becomes a system of Mathieu equations \cite{mathieu}, which are a particular case of the Hill equations. The advantage of this choice is that much more precise
information is known on the behavior of the stability regions. Exploiting this fact we give a detailed explanation of how the stability
is lost for the dominating mode of \eq{hamiltonian} and which residual mode first captures its energy. Notice that by \cite{argaz} we know that several different choices of $U$ yield a similar response in the bridge. \par
The paper is organized as follows. In Section \ref{s2} we state the theoretical criterion governing the energy transfer between modes and in Section \ref{somenumres} we collect several numerical experiments which confirm and illustrate the theoretical results. In Section \ref{different} we discuss different choices of the potential $U$. Finally, in Section \ref{mec} we give a mechanical interpretation of the numerical results and we suggest some structural remedies to prevent instability in suspension bridges.

\section{Main results: energy dependent stability}\label{s2}
For $\mu$, $\lambda_1$, $\lambda_2$ being positive real numbers, $x_0\in\R\setminus\{0\}$ and $\eps>0$, we consider the following problem
\begin{equation}\label{PG}
\begin{cases}
\ddot{y}+\mu^2 y + U_{y}(y,z_1,z_2)=0   &\ y(0)=x_0, \ \dot{y}(0)=0\\
\ddot{z}_1+\lambda_1^2 z_1+ U_{z_1}(y,z_1,z_2)=0 &\ z_1(0)=\eps x_0, \ \dot{z}_1(0)=0\\
\ddot{z}_2+\lambda_2^2 z_2+ U_{z_2}(y,z_1,z_2)=0 &\ z_2(0)=\eps x_0, \ \dot{z}_2(0)=0\,,\\
\end{cases}
\end{equation}
where $U:\R^3\to\R$ is a non-negative, differentiable function with locally Lipschitz derivatives, and $U_{y}$, $U_{z_i}$ stand for its partial derivatives
with respect to $y$ and $z_i$, respectively.
The conserved total energy of \eq{PG} is given by
\begin{equation}\label{energyPG}
E:=\frac{\dot{y}^2}{2}+\frac{\dot{z_1}^2}{2}+\frac{\dot{z_2}^2}{2}+\frac{\mu^2}{2}y^2+\frac{\lambda_1^2}{2} z_1^2+\frac{\lambda_2^2}{2} z_2^2+U(y,z_1,z_2)\,.
\end{equation}
\par Along the paper, we mainly deal with the potential
\neweq{quadratic}
U(y,z_1,z_2)=\frac{y^2 z_1^2+y^2z_2^2+z_1^2z_2^2}{2}\,,
\endeq
see Section \ref{different} for a discussion about different choices. With the potential $U$ as in \eq{quadratic}, system \eq{PG} becomes
\neweq{hamilton2}
\left\{\begin{array}{ll}
\ddot{y}+\mu^2 y+(z_1^2+z_2^2)y=0\quad &\ y(0)=x_0, \ \dot{y}(0)=0\\
\ddot{z}_1+\lambda_1^2 z_1+(y^2+z_2^2)z_1=0\quad &\ z_1(0)=\eps x_0, \ \dot{z}_1(0)=0\\
\ddot{z}_2+\lambda_2^2 z_2+(y^2+z_1^2)z_2=0\quad &\ z_2(0)=\eps x_0, \ \dot{z}_2(0)=0\,.
\end{array}\right.
\endeq

If in \eq{hamilton2} we take $\eps=0$ (and $x_0\neq0$), then its unique solution satisfies
$z_1\equiv z_2\equiv0$, while $y$ solves $\ddot{y}+\mu^2 y=0$. Notice that, up to a time translation, any initial condition $(y(0),\dot{y}(0))\neq(0,0)$, yields the same solution as $(y(0),\dot{y}(0))=(x_0,0)$ for some $x_0$. Therefore, for $\eps=0$ (and $x_0\neq0$), system \eq{hamilton2} admits the unique solution $(\bar y,0,0)=(x_0\cos(\mu t),0,0)$ and the conserved energy
\begin{equation}\label{E-lin}
E:=\frac{\dot{y}^2}{2} + \frac{\mu^2}{2}y^2 = \frac{\mu^2}{2}x_0^2.
\end{equation}
Since our aim is to study the behavior of solutions for small $\eps$, we linearize the $z_i$ equations of \eq{hamilton2} around this solution and we obtain the following system of Mathieu equations \cite{mc}
\begin{equation}\label{Mathieu system}
\begin{cases}
\ddot{\xi}_1+\left(\lambda_1^2+\frac{x_0^2}{2} +\frac{x_0^2}{2}  \cos(2 \mu t)\right) \xi_1=0 &\\
\ddot{\xi}_2+\left(\lambda_2^2+\frac{x_0^2}{2} +\frac{x_0^2}{2}  \cos(2 \mu t)\right) \xi_2=0\,. &\\
\end{cases}
\end{equation}
By a change of variables (without renaming the $\xi_i$'s), we may rewrite the equations in \eq{Mathieu system}
in the canonical form:

\begin{equation}\label{Mathieu canon}
\displaystyle{\ddot{\xi_i}+\left( \alpha_i+2q_i \cos(2 t) \right) \xi_i=0}\,, \quad  \text{for } i=1,2\,,
\end{equation}
with
\neweq{retteparametriche}
\alpha_i(x_0)=\frac{2\lambda_i^2+x_0^2}{2\mu^2} \quad \text{and} \quad q_i(x_0)=q(x_0)=\frac{x_0^2}{4\mu^2}\ \quad \text{for } i=1,2\,,
\endeq
so that
\neweq{rette}
\alpha_i(q)=\frac{\lambda_i^2}{\mu^2}+2q\,.
\endeq

Let us explain what we mean by stability for system \eq{hamilton2}.

\begin{definition}\label{def:stable}
The solution $(\bar y,0,0)=(x_0\cos(\mu t),0,0)$ to system \eqref{hamilton2} for $\eps=0$ is said to be stable if the trivial solutions $\xi_i\equiv0$ ($i=1,2$) of
\eqref{Mathieu canon} are both stable. In the other cases, $(\bar y,0,0)$ is said to be unstable.
\end{definition}

Note that the two equations in \eqref{Mathieu canon} are uncoupled and therefore the trivial solution $(\xi_1,\xi_2)=(0,0)$ is stable if and only if
both the trivial solutions $\xi_1=0$ and $\xi_2=0$ of each equation in \eqref{Mathieu canon} are stable. The numerical results described in Section \ref{somenumres} confirm that this definition is well suited to characterize the instability.
As we shall see, the stability of $(\bar y,0,0)$ depends on its energy \eq{E-lin}. Therefore, the following definition will be useful.

\begin{definition}\label{activating energies}
We say that the energy $E$ in \eqref{E-lin} is activating for the residual mode $z_i$ ($i=1$ or $i=2$) of system \eqref{hamilton2} if the trivial
solution $\xi_i\equiv 0$ of the Mathieu equation \eqref{Mathieu canon} is unstable. Otherwise, we say that it is non-activating.
\end{definition}

We may now state and prove the following stability result.

\begin{theorem}\label{stability}
Let $\mu,\lambda_1,\lambda_2>0$ and $x_0\in\R\setminus\{0\}$. Let $E>0$ be the energy \eqref{E-lin} associated to the solution
$(\bar y,0,0)=(x_0\cos(\mu t),0,0)$ to system \eqref{hamilton2} for $\eps=0$. For each $i=1,2$ there exists an increasing divergent sequence $\{E^i_m\}_{m=0}^\infty$ such that $E^i_0=0$ and
\par (i) $E$ is non-activating whenever $E\in(E_{2k}^i,E_{2k+1}^i)$ for some $k\ge0$;
\par (ii) $E$ is activating whenever $E\in(E_{2k+1}^i,E_{2k+2}^i)$ for some $k\ge0$.
\end{theorem}
\begin{proof} 
We first recall that, given $q >0$, the Mathieu equation
$$
\ddot{w}+\left(a+2q \cos(2 t)\right) w=0
$$
admits solutions which are either $\pi$ or $2\pi$-periodic only if $a$ belongs to the countably infinite sets of the so-called Mathieu characteristic values $\{a_n(q)\}_{n\geq 0}$ and $\{b_n(q)\}_{n\geq 1}$, see \cite{AbSt,mc,verh}. The characteristic curves do not intersect, that is, we have
\neweq{ordinate}
a_0(q)<b_1(q)<a_1(q)<\dots<b_n(q)<a_n(q)<b_{n+1}(q)<\dots\qquad\forall n\ge2\, .
\endeq
Moreover, their asymptotic behavior for large $q$ is
\neweq{qinfty}
a_n(q)\sim-2q\, ,\quad b_n(q)\sim-2q\qquad\mbox{as }q\to\infty\, ,
\endeq
while for small $q$ we have
\neweq{qzero}
\left\{\begin{array}{ll}
a_0(q)=o(q)\,,\quad b_1(q)=1-q+o(q)\,,\quad a_1(q)=1+q+o(q)\,,\\
b_n(q)=n^2+o(q)\ \mbox{ and }\ a_n(q)=n^2+o(q)\quad \forall n\ge2\,,
\end{array}\right.\qquad\mbox{as }q\to0\,,
\endeq
see \cite[Sections 2.151 and 12.30]{mc}.\par
The characteristic curves $a_n(q)$ and $b_n(q)$ divide the $(q,a)$-plane into stable and unstable regions, see the left picture in Figure \ref{calcoloenergia}, where the red lines correspond to the characteristic curves.
Denote with $S_n$ ($n\geq 0$) the stability (white) regions and with $U_n$ ($n\geq 1$) the instability (gray) regions. For $n\geq 0$ we have
$$S_n:=\{(q,a): q>0\,, a_n(q)<a<b_{n+1}(q)\},$$
while for $n\geq 1$ we have
$$U_n:=\{(q,a): q>0\,, b_n(q)<a<a_{n}(q)\}.$$
\begin{figure}[ht]
\begin{center}
{\includegraphics[height=60mm, width=120mm]{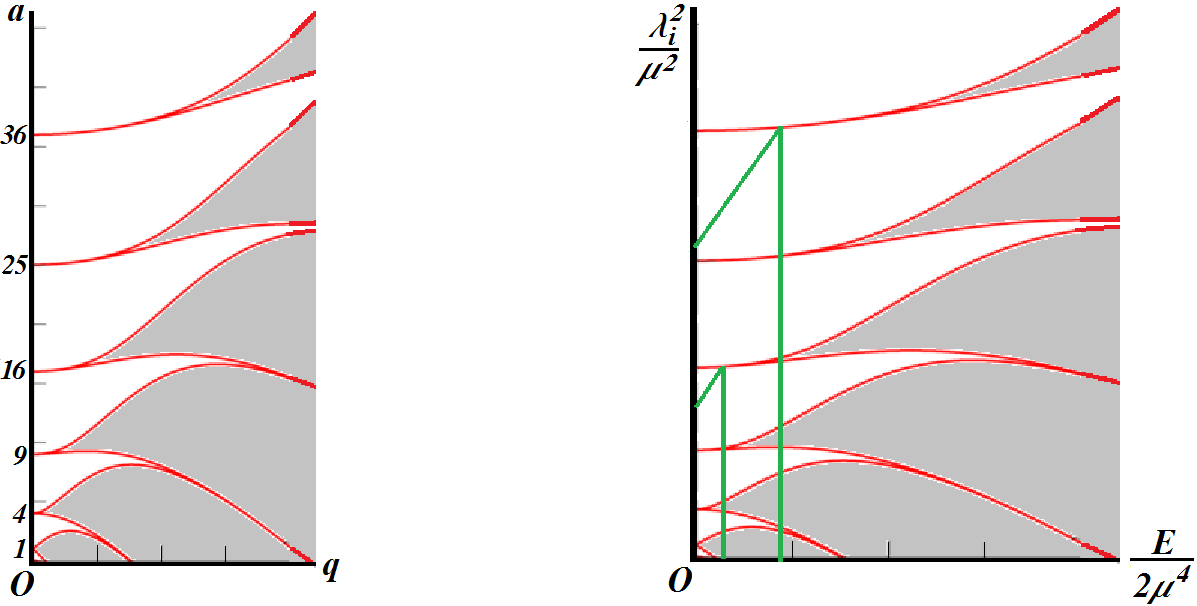}}
\caption{On the left the Mathieu diagram, on the right how to compute the energy threshold. The instability regions are gray.}\label{calcoloenergia}
\end{center}
\end{figure}

To each couple $(\mu,\lambda_i)$ in system \eq{hamilton2} we associate the sequence of energies $\{E^i_m\}_{m=0}^\infty$ as follows. The energy associated to $(\bar{y},0,0)$ satisfies
\eq{E-lin}, that is
\neweq{nuovaenergia}
E=\frac{\mu^2 x_0^2}{2} = 2\mu^4 q\,,
\endeq
where the second equality is due to \eq{retteparametriche}. Hence, as $E$ increases from $E=0$ to $E=\infty$ the parameters $(q(x_0),\alpha_i(x_0))$ in \eq{Mathieu canon} move
along the line \eq{rette} in the $(q,a)$-plane, according to the law \eq{retteparametriche}. We need to study the intersections of these lines with $S_n$ and $U_n$.\par
If $\alpha_i(0)=\lambda_i^2/\mu^2\in(n^2,(n+1)^2)\equiv(a_n(0),b_{n+1}(0))$ for some $n=0,1,2,\dots$ then, since all the functions involved are continuous, there
exists $E_1^i>0$ such that $a_n(q)<\alpha_i(q)<b_{n+1}(q)$ for all $E<E_1^i$, that is, for all $q>0$ sufficiently small in view of \eq{nuovaenergia}.
If $\alpha_i(0)=\lambda_i^2/\mu^2=n^2=a_n(0)<(n+1)^2=b_{n+1}(0)$ for some $n=1,2,\dots$ then, since the lines \eq{rette} have slope 2 and since \eq{qzero} holds,
we conclude again that $a_n(q)<\alpha_i(q)<b_{n+1}(q)$ for all $E>0$ sufficiently small. Therefore,
\neweq{firstinterval}
\exists n\in\N\, ,\quad \exists E_1^i>0\ \mbox{ s.t. }\ (q,\alpha_i(q))\in S_n\qquad\forall E<E_1^i\, .
\endeq
The largest possible value of $E_1^i$ may be determined as follows: one finds the abscissa $q$ of the intersection between $\alpha_i(q)$ and $b_{n+1}(q)$ where $n$ is as in \eq{firstinterval} (see the
corner of the green line in the right picture of Figure \ref{calcoloenergia}), then one computes $E_1^i$ according to \eq{nuovaenergia}. The asymptotic
estimate \eq{qinfty} ensures that $E_1^i<\infty$.\par
By \eq{ordinate}-\eq{qinfty}-\eq{qzero} we infer that the straight line \eq{rette} intersects at least once each characteristic curve $a_n$ and $b_n$ provided that
$n>\lambda_i/\mu$; moreover, at each crossing, the line moves from some $U_n$ to $S_n$ or from some $S_n$ to $U_{n+1}$, thereby alternating its intersection
with gray and white regions in Figure \ref{calcoloenergia}. Since the stability of the trivial solution $\xi_i\equiv0$ of \eq{Mathieu canon} depends on the
position of $(q,\alpha_i)$ in the Mathieu diagram, this completes the proof of the theorem.\end{proof}

Theorem \ref{stability} states, in particular, that the first energy interval $(0,E_1^i)$ is non-activating for both $i=1,2$. We may rephrase this property
as follows.

\begin{corollary}\label{thm:stability}
For every triple of real positive parameters $(\mu,\lambda_1,\lambda_2)$ there exists an energy $\bar{E}=\bar{E}(\mu,\lambda_1,\lambda_2)>0$ such that the
solution $(\bar y,0,0)$ to system \eqref{hamilton2} for $\eps=0$ is stable provided that its conserved energy $E$ defined in \eqref{E-lin} satisfies $E\leq \bar{E}$.
\end{corollary}

By combining Theorem \ref{stability} with Definition \ref{activating energies} we obtain the following theoretical criterion to determine which residual
mode captures the energy of the dominating mode $y$:

\begin{corollary}\label{beyond stability}
Let $E>0$ be the energy \eqref{E-lin} of system \eqref{hamilton2} associated to the solution $(\bar{y},0,0)$ for $\eps=0$. If $E\in (E^i_{2k+1},E^i_{2k+2})$ for some $k\ge0$
and for $i=1$ or $i=2$, then the residual mode $z_i$ captures the energy of the dominating mode $y$.
\end{corollary}

As we shall see in Section \ref{somenumres} it may happen that both the residual modes capture the energy. Furthermore, the amount of captured energy
depends on how far is the point $(q,a)$ from the stability region. Therefore, the amplitude of the corresponding activating interval plays an
important role. In Section \ref{somenumres} we shall see that if it is sufficiently small 
then there is no ``visible'' activation, since the crossing through the unstable region
is ``too fast''.


\section{Numerical results}\label{somenumres}

We consider again system \eq{hamilton2}. For $\eps$ small, its conserved energy is given by
\neweq{Ex0}
E=\frac{1}{2}\left( \dot{y}^2+\dot{z_1}^2+\dot{z_2}^2+\mu^2 y^2+\lambda_1^2z_1^2+\lambda_2^2z_2^2+y^2z_1^2+y^2z_2^2+z_1^2z_2^2\right)
\approx\frac{\mu^2}{2}\,x_0^2\,.
\endeq

From the proof of Theorem \ref{stability} we learn that the activating intervals for the energy can be computed by determining for which values of $q$ the couple $(q,\alpha_i)$ in \eq{retteparametriche} lies in the instability regions $U_n$ with $n^2>\alpha_i(0)$, namely by intersecting the lines \eq{rette} with the characteristic curves of the Mathieu equations. A numerical approximation of the intersection points can be obtained with Mathematica, by using the functions
\begin{center}
{\tt MathieuCharacteristicA[n,x]} \ and \ {\tt MathieuCharacteristicB[n,x]}\,.
\end{center}
In turn, by \eq{retteparametriche}, this intersection yields the initial data $x_0$ for which the energy belongs to the activating intervals. In the experiments below we plot the solutions to \eq{hamilton2} for suitable choices of the parameters $\mu,\lambda_1, \lambda_2$ and for different values of the initial data $x_0$.

\subsection{Experiment 1}\label{exp1}
Fix $\mu^2=1$, $\lambda_1^2=0.1$, $\lambda_2^2=0.9$ and $\eps=10^{-3}$. By computing, as explained above, the intersection points of the characteristic curves of the Mathieu equations $b_1<a_1<b_2<a_2$ with the straight lines: 
\neweq{above}
(\ell_1)\quad a=0.1+2q\qquad\mbox{and}\qquad (\ell_2)\quad a=0.9+2q\, 
\endeq
and thanks to \eq{retteparametriche}, we obtain that the couple $(q,\alpha_1(q))$, as given in \eq{rette}, lies in the instability region $U_1$ (resp. $U_2$) if $x_0$ belongs to the interval $I_1^1=(1.1,1.8)$ (resp. $I_1^2=(2.69,3.44)$). For these choices of $x_0$ the energy \eq{Ex0} is activating for $z_1$. Similarly, if $x_0$ belongs to the interval $I_2^1=(0.36,0.63)$ (resp. $I_2^2=(2.42,2.99)$), then the couple $(q,\alpha_2(q))$ lies in the instability region $U_1$ (resp. $U_2$) and the corresponding energy \eq{Ex0} is activating for $z_2$. \par
With Mathematica we plot the graphs of the solution of \eq{hamilton2} on the interval of time $t\in[0,400]$ for varying $x_0$ (and, therefore, varying $E$) close to the intervals determined above. We varied $x_0$ from $x_0=0.1$ to $x_0=3$ with step $0.1$; we obtained plots of the residual modes $z_1$ and $z_2$ and we could see which
of the two modes (if any) captured the energy of the dominating mode $y$. We also plotted the graph of $y$ which is somehow less interesting since for small $t>0$ it essentially
looks like $y(t)\approx x_0\cos(\mu t)$ and is too large to allow to see the variations of the residual modes $z_i$. Since both the $z_i$ start
with amplitude of oscillations of the order of $10^{-3}$ (or even $10^{-4}$ for small $x_0$), we could detect their instability when their oscillations
increased in amplitude of at least one order of magnitude. In order not to plot too many pictures, we describe the obtained results with 15 graphs
from $x_0=0.2$ to $x_0=3$ with step $0.2$. All the graphs are complemented with comments.\par

\begin{figure}[ht]
\begin{center}
{{\includegraphics[height=27mm, width=50mm]{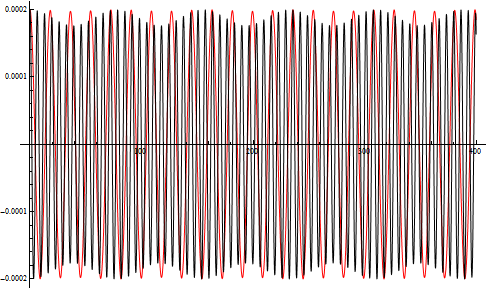}}\ {\includegraphics[height=27mm, width=50mm]{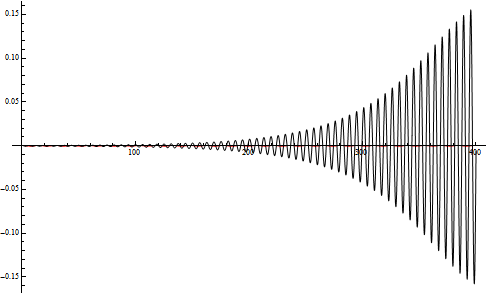}}\ {\includegraphics[height=27mm, width=50mm]{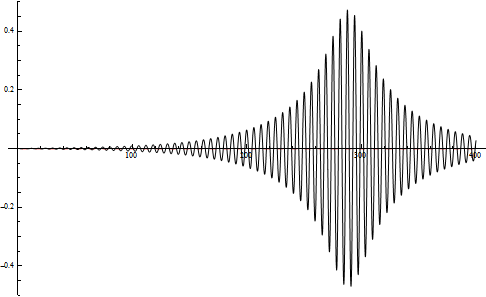}}}
\caption{Plots of $z_1$ (red) and $z_2$ (black) for $x_0\in\{0.2,0.4,0.6\}$ (left to right).}\label{primaterna}
\end{center}
\end{figure}

In Figure \ref{primaterna} we display the plots for $x_0\in\{0.2,0.4,0.6\}$. It is apparent that for $x_0=0.2$ both the residual modes remain small, nearly as their initial amplitude. It is however already visible that
$z_2$ (black) has somehow regular cycles of variable amplitude. For $x_0=0.4$ we only see $z_2$ which grows up to $\approx0.16\gg z_2(0)$ while $z_1$
is not visible because it remains of the order of $z_1(0)$; this picture shows that $z_2$ has captured some of the energy of $y$ whose amplitude has
decreased as in Figure \ref{typical}. The same phenomenon is accentuated for $x_0=0.6$ where it appears earlier in time and $z_2$ grows up
until $\approx0.43$.\par
Let us analyze these results with the aid of the theoretical results of Section \ref{s2}. We enlarge the diagram
of the instability curves of the Mathieu equations and, on the same graph, we plot the straight lines
$\ell_1$ and $\ell_2$ as defined in \eq{above}. Since we are in the region where $a\le1$, the obtained picture on the interval $q\in[0,1/4]$ is represented in Figure \ref{enlarged}.

\begin{figure}[ht]
\begin{center}
{\includegraphics[height=53mm, width=121mm]{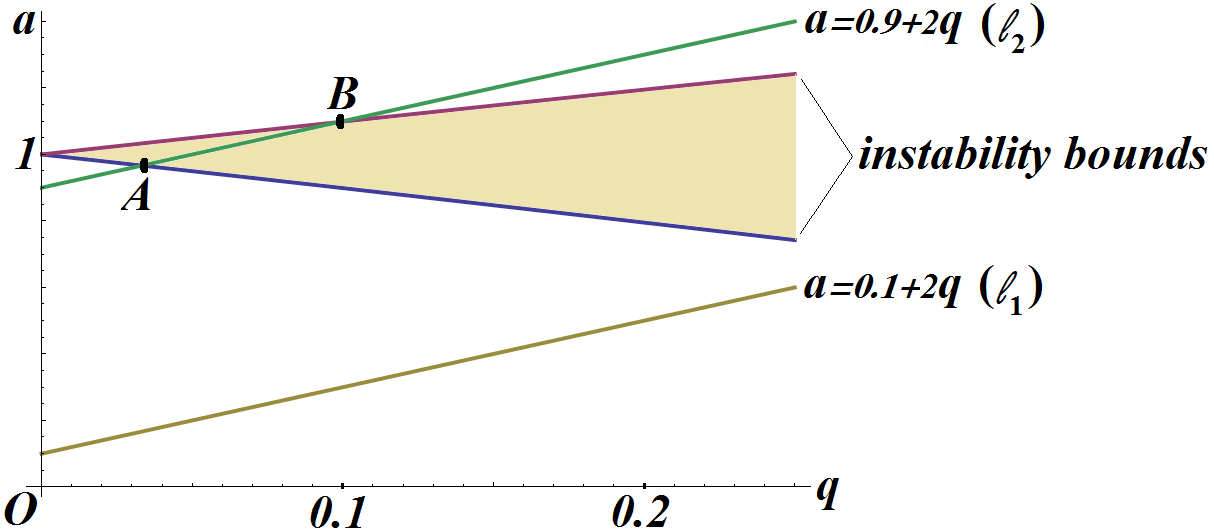}}
\caption{Intersections between the stability regions and the parametric lines (local view).}\label{enlarged}
\end{center}
\end{figure}
\noindent
Starting from $q=0$ (that is, $x_0=0$), the line $(\ell_2)$ is the first one which exits the (white) stability region. This happens at the point $A$
which, again computed with Mathematica, has the abscissa $q\approx0.033$ and therefore, in view of \eq{retteparametriche}, $x_0\approx0.36$, i.e. the left endpoint of the interval $I_2^1$. At this
amplitude of oscillation of $y$, in accordance with our theoretical results, we see that the residual mode $z_2$ starts capturing its energy. Figure \ref{primaterna} confirms that the transition occurs
for $0.2<x_0<0.4$. In view of \eq{Ex0}, the critical energy is $E\approx0.066$. \par
For larger $x_0$, that is $x_0\in\{0.8,1,1.2\}$, we obtained the plots in Figure \ref{secondaterna}.
\begin{figure}[ht]
\begin{center}
{{\includegraphics[height=27mm, width=50mm]{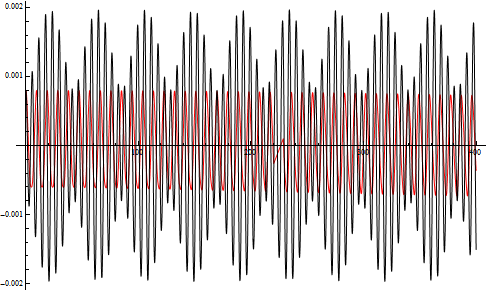}}\ {\includegraphics[height=27mm, width=50mm]{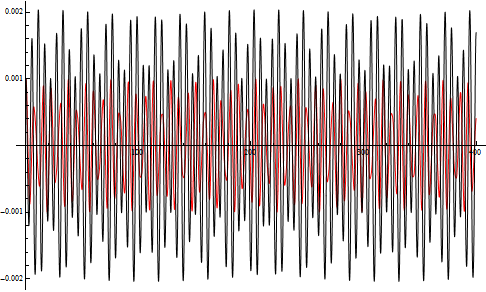}}\ {\includegraphics[height=27mm, width=50mm]{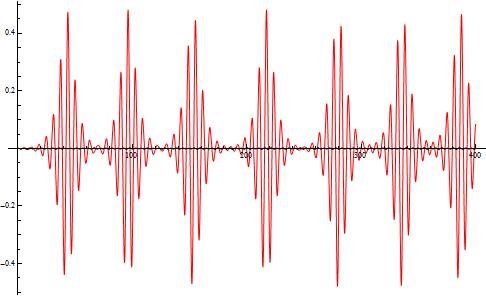}}}
\caption{Plots of $z_1$ (red) and $z_2$ (black) for $x_0\in\{0.8,1,1.2\}$ (left to right).}\label{secondaterna}
\end{center}
\end{figure}
In the first two pictures ($x_0\in\{0.8,1\}$) we see that none between $z_1$ and $z_2$ captures the energy of $y$, they essentially remain of the same order
of magnitude as the initial data. This means that the line $(\ell_1)$ has not yet entered in the instability region of the Mathieu diagram while $(\ell_2)$
has exited. Looking again at Figure \ref{enlarged}, we see that the latter fact occurs at the point $B$ corresponding to $q\approx0.099$ and therefore
to $x_0\approx0.63$, i.e. the right endpoint of the interval $I_2^1$. At this amplitude of oscillation of $y$, the residual mode $z_2$ stops capturing its energy. Figures \ref{primaterna} and
\ref{secondaterna} confirm that the transition occurs for $0.6<x_0<0.8$. Namely, the activating interval numerically observed is the one determined by the theoretical results. Moreover, Figure \ref{enlarged} also shows that $(\ell_1)$ has not yet
entered in the instability region: in order to see when this happens we have to take a larger view of the Mathieu diagram, see Figure \ref{general}.
\begin{figure}[ht]
\begin{center}
{\includegraphics[height=53mm, width=127mm]{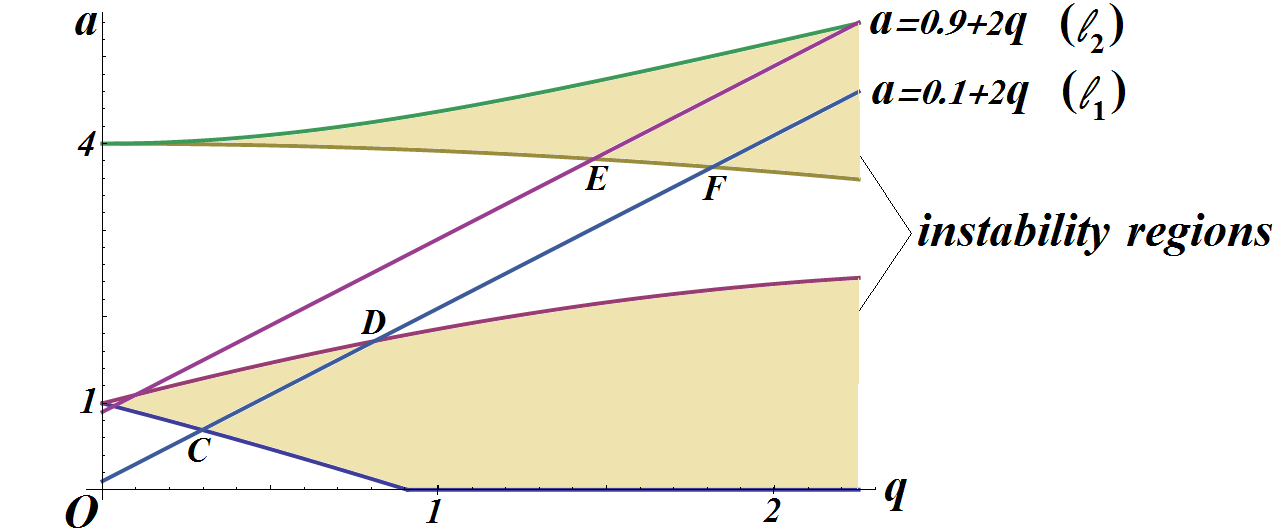}}
\caption{Intersections between the stability regions and the parametric lines (global view).}\label{general}
\end{center}
\end{figure}
In this picture we represent the diagram for $q\in[0,9/4]$ since $q=9/4$ corresponds to $x_0=3$; moreover, we do not place again the points $A$ and $B$
in order to have a more readable picture. The point where $(\ell_1)$ enters the instability region is $C$, see Figure \ref{general}: numerically, it
corresponds to $q\approx0.3$ and to $x_0\approx1.1$ (left endpoint of $I_1^1$). This explains why in Figure \ref{secondaterna}, case $x_0=1.2$, we see that $z_1$ enlarges and
captures the energy of the dominating mode $y$.\par
By increasing further $x_0$, that is, $x_0\in\{1.4,1.6,1.8\}$ we obtained the plots in Figure \ref{terzaterna}.

\begin{figure}[ht]
\begin{center}
{{\includegraphics[height=27mm, width=50mm]{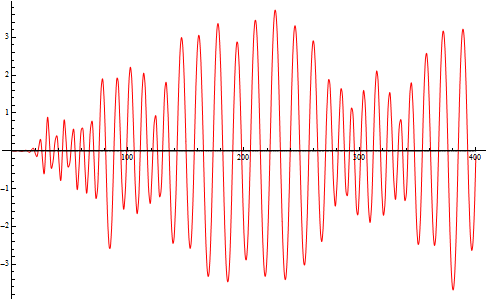}}\ {\includegraphics[height=27mm, width=50mm]{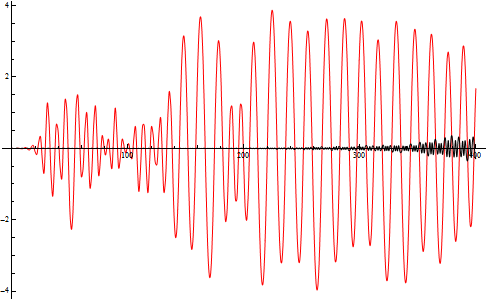}}\ {\includegraphics[height=27mm, width=50mm]{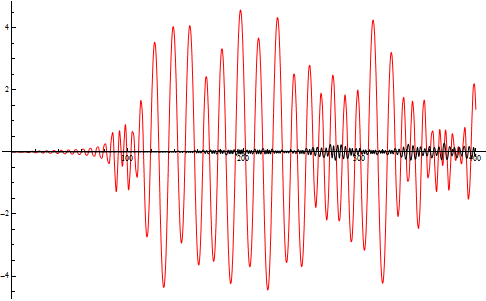}}}
\caption{Plots of $z_1$ (red) and $z_2$ (black) for $x_0\in\{1.4,1.6,1.8\}$ (left to right).}\label{terzaterna}
\end{center}
\end{figure}
\noindent
We see here that $z_1$ may become even larger than $x_0$, that is, of the initial amplitude of the dominating mode. From the energy conservation
we infer that this can happen only if $y$ is almost 0 when $|z_1|$ reaches its maximum. This shows that there has been a change of the frequencies
and that the period of $z_1$ is a multiple (possibly the same) of the period of $y$. For $x_0\in\{1.6,1.8\}$ we see that also $z_2$ increases its
amplitude after some (long)
interval of time. We believe that this happens because $z_2$ captures some energy from $z_1$; this would mean that the linearized problem has changed
and that different straight lines should be drawn on the Mathieu diagram. Therefore, this {\em does not} mean that $(\ell_2)$ has reached the point
$E$ in Figure \ref{general}.\par
For $x_0\in\{2,2.2,2.4\}$ we obtained the plots in Figure \ref{quartaterna}.

\begin{figure}[ht]
\begin{center}
{{\includegraphics[height=27mm, width=50mm]{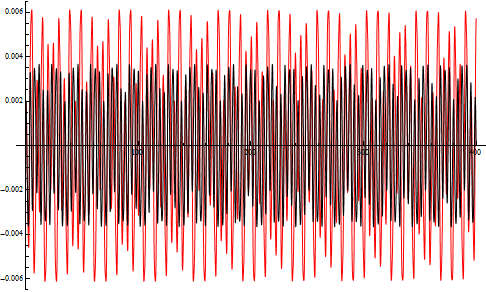}}\ {\includegraphics[height=27mm, width=50mm]{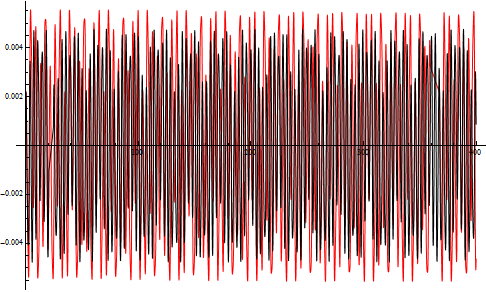}}\ {\includegraphics[height=27mm, width=50mm]{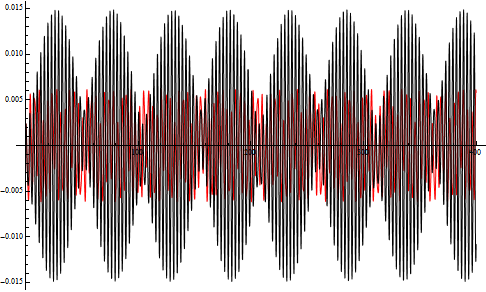}}}
\caption{Plots of $z_1$ (red) and $z_2$ (black) for $x_0\in\{2,2.2,2.4\}$ (left to right).}\label{quartaterna}
\end{center}
\end{figure}
\noindent
If $x_0\in\{2,2.2\}$ we see that no residual mode is capturing the energy of the dominating mode, both $z_1$ and $z_2$ have an amplitude of oscillation
of the order of $10^{-3}$. This means that the line $(\ell_1)$ has crossed the point $D$ which, numerically, is seen to occur for $q\approx0.81$ and to
$x_0\approx1.8$ (right endpoint of $I_1^1$). This fact is confirmed by a finer experiment performed for $x_0=1.81$: in this case, the picture looks like the left one in Figure
\ref{quartaterna}. If $x_0=2.4$, from Figure \ref{quartaterna} we see that $z_2$ starts to become larger, which means that the line $(\ell_2)$ is approaching
the point $E$ in Figure \ref{general}. And, indeed, we numerically found that the abscissa of $E$ is $q\approx1.46$ which corresponds to $x_0\approx2.42$ (left endpoint of $I_2^2$).\par
For $x_0\in\{2.6,2.8,3\}$ we obtained the plots in Figure \ref{quintaterna}.

\begin{figure}[ht]
\begin{center}
{{\includegraphics[height=27mm, width=50mm]{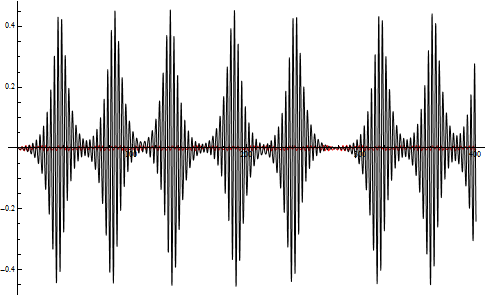}}\ {\includegraphics[height=27mm, width=50mm]{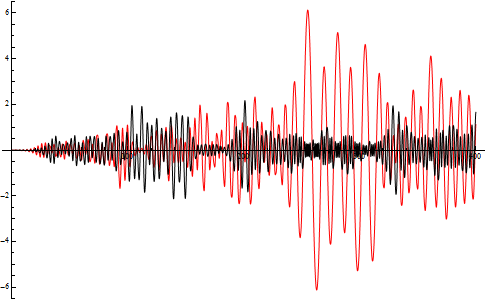}}\ {\includegraphics[height=27mm, width=50mm]{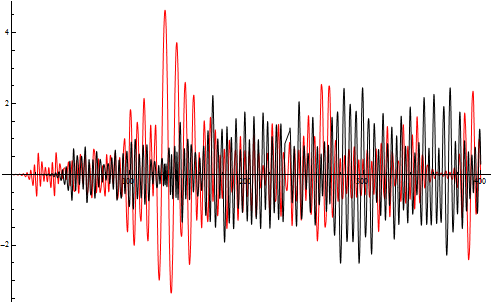}}}
\caption{Plots of $z_1$ (red) and $z_2$ (black) for $x_0\in\{2.6,2.8,3\}$ (left to right).}\label{quintaterna}
\end{center}
\end{figure}
\noindent
For $x_0=2.6$ the line $(\ell_2)$ is beyond $E$ and has entered in the second instability region, a fact which is clearly displayed by the left
picture in Figure \ref{quintaterna}. The point $F$ in Figure \ref{general} is the point where also $(\ell_1)$ enters in the second instability region:
its abscissa is $q=1.81$ corresponding to $x_0\approx2.69$ (left endpoint of $I_1^2$). And indeed, the plots for $x_0\in\{2.8,3\}$ essentially show a
chaotic behavior where both the residual modes capture the energy of the dominating mode.\par
The just described numerical results enable us to give a precise answer to the question raised in the title relatively to the particular second
order Hamiltonian system \eq{hamilton2}. 
\begin{quote}\textit{Which residual mode captures the energy of the dominating mode depends on the amplitude of oscillation
or, equivalently, on the amount of energy present within \eq{hamilton2}. }
\end{quote}
The response is summarized in the following table where RMCE means
{\em residual mode capturing the energy} and $x_0$ varies in the interval $[0,3]$.

\begin{table}[htbp]
\caption{residual mode capturing the energy (RMCE) when $\mu^2=1,$ $\lambda_1^2=0.9$ and $\lambda_2^2=0.1$.} 
\begin{center}
\begin{tabular}{|c|c|c|c|c|c|c|c|}
\hline
$x_0\in$ & $[0,0.36)$ & $I_2^1$ & $(0.63,1.1)$ & $I_1^1$ & $(1.8,2.42)$ & $I_2^2\setminus I_1^2$ & $I_2^2\cap I_1^2$ \\
\hline
RMCE & none & $z_2$ & none & $z_1$ & none & $z_2$ & both \\
\hline
\end{tabular}
\end{center}
\end{table}

\subsection{Experiment 2}

Consider system \eq{hamilton2} with $\mu=1$, $\lambda_1=2$, $\lambda_2=4$ and $\eps=10^{-3}$. We proceed as in Experiment 1. The straight lines \eq{rette} in this case are
$$(\ell_1)\quad a=4+2q\qquad\mbox{and}\qquad (\ell_2)\quad a=16+2q\,. $$

To determine the first two activating intervals for each of the residual modes, we first intersect $\ell_1$ with the characteristic curves $b_3<a_3<b_4<a_4$ and then $\ell_2$ with the characteristic curves $b_5<a_5<b_6<a_6$. With Mathematica and recalling \eq{retteparametriche}, we obtain that the couple $(q,\alpha_1(q))=(q,4+2q)$ lies in the instability region $U_3$ (resp. $U_4$) if $x_0$ belongs to the interval $I_1^1=(3.22,3.42)$ (resp. $I_1^2=(5.08,5.42)$). For these choices of $x_0$ the energy \eq{Ex0} is activating for $z_1$. The couple $(q,\alpha_2(q))=(q,16+2q)$ lies in the instability region $U_5$ (resp. $U_6$) if $x_0$ belongs to the interval $I_2^1=(4.349,4.357)$ (resp. $I_2^2=(6.58,6.614)$) and the energy \eq{Ex0} is activating for $z_2$. As in Section \ref{exp1}, we have plotted the graphs of $z_1$ and $z_2$ for many choices of $x_0$ both inside and outside the above intervals $I_i^j$. For $x_0$ entering in the intervals $I_1^1$ and $I_1^2$, the behavior of the solutions $z_1$ and $z_2$ is as in Figure \ref{primaterna} with $z_1$ and $z_2$ swapped. The amplitude of the oscillations of $z_1$ increases by a factor of 10 when crossing $I_1^1$ and by a factor of 4 when crossing $I_1^2$. If $x_0$ belongs to $I_2^1$ the energy transfer on $z_2$ cannot be noticed. The reason is the small amplitude of the interval $I_2^1$ (of order $< 10^{-2}$); in other words, for small energies $z_2$ appears more stable than $z_1$. Finally, if $x_0$ belongs to the interval $I_2^2$ the energy transfer on $z_2$ can hardly be noticed, since the amplitude of the oscillations of $z_2$ increases by a factor of 2 when crossing $I_2^2$. The results are summarized in Table 2.
\begin{table}[htbp]
\caption{residual mode capturing the energy (RMCE) when $\mu^2=1,$ $\lambda_1^2=4$ and $\lambda_2^2=16$.}
\begin{center}
\begin{tabular}{|c|c|c|c|c|c|c|c|c|}
\hline
$x_0\in$ & $[0,3.22)$ & $I_1^1$ & $(3.42,4.349)$ & $I_2^1$ & $(4.357,5.08)$ & $I_1^2$ & $(5.42,6.58)$&$I_2^2$ \\
\hline
RMCE & none & $z_1$ & none & none & none & $z_1$ & none & $z_2$/none \\
\hline
\end{tabular}
\end{center}
\end{table}

\subsection{Experiment 3}
We fix $\mu=\sqrt{2}/2$, $\lambda_1=2$, $\lambda_2=4$ and $\eps=10^{-2}$, namely we double the ratios $\frac{\lambda_i^2}{\mu^2}$ of Experiment 2. Here the straight lines \eq{rette} become
$$(\ell_1)\quad a=8+2q\qquad\mbox{and}\qquad (\ell_2)\quad a=32+2q\,. $$
With Mathematica, we intersect $\ell_1$ with the characteristic curves $b_3<a_3<b_4<a_4$ and $\ell_2$ with the characteristic curves $b_6<a_6<b_7<a_7$. By \eq{retteparametriche}, arguing as in the previous experiments, we obtain that if $x_0$ belongs to the intervals $I_1^1=(1.007,1.009)$ and $I_1^2=(2.915,2.969)$, then the energy \eq{Ex0} is activating for $z_1$. If $x_0$ belongs to the intervals $I_2^1=(2.01467,2.01468)$ and $I_2^2=(4.2233,4.2239)$, then the energy \eq{Ex0} is activating for $z_2$. The behavior of both $z_1$ and $z_2$ becomes much more stable and we have to wait until the second activating interval for $z_1$, namely $I_1^2$, to register the first significant energy transfer on a residual mode. Table 3 summarizes what we numerically observed.

\begin{table}[htbp]
\caption{Residual mode capturing the energy (RMCE) when $\mu^2=1/2,$ $\lambda_1^2=4$ and $\lambda_2^2=16$.} 
\begin{center}
\begin{tabular}{|c|c|c|c|c|c|c|c|c|}
\hline
$x_0\in$ & $[0,1.007)$ & $I_1^1$ & $(1.009,2.01467)$ & $I_2^1$ & $(2.01468,2.915)$ & $I_1^2$ & $(2.969,4.2233)$&$I_2^2$ \\
\hline
RMCE & none & none & none & none & none & $z_1$ & none & none \\
\hline
\end{tabular}
\end{center}
\end{table}

\subsection{Conclusions from the numerical results}
We performed further experiments which confirmed the just illustrated precise pattern. The lines \eq{rette} intersect alternatively the stability/instability regions giving rise to one of the above pictures. Furthermore, the observed activating intervals coincide with those expected from our theoretical results. Summarizing, we may draw the following conclusions.\par
$\bullet$ Which residual mode first captures the energy of the dominating mode depends on the {\em ratios} $\lambda_i/\mu$: these ratios determine the point
of the $a$-axis in the Mathieu diagram where the straight lines \eq{retteparametriche} start at zero energy.\par
$\bullet$ The energy threshold for instability is $\mu^2x_0^2/2$, see \eq{nuovaenergia}, and one can use Figure \ref{calcoloenergia} to compute it.\par
$\bullet$ The residual modes grow up earlier in time and wider in amplitude if $x_0$ is such that the corresponding parameters $(q,\alpha_i)$ in \eq{retteparametriche}
are far from the stability region, see the last two pictures in Figure \ref{primaterna}.\par
$\bullet$ When the quotient $\frac{\lambda_i^2}{\mu^2}$ increases, the residual modes display a very stable behavior. A theoretical explanation of this fact comes from the classical stability theory for the Mathieu equation. Indeed, it can be proved that for $a\gg q>0$, corresponding in our case to $\frac{\lambda_i^2}{\mu^2}$ large and $E$ small, the trivial solution of the Mathieu equation is stable, see \cite[Section 4.80]{mc}.\par
$\bullet$ If the amplitude of the activating interval for the energy of residual mode $z_i$ is small, then there is no ``visible'' activation, see Tables 2 and 3 and use \eq{Ex0} to obtain the response in terms of the energy.

\section{Different potentials}\label{different}

It is quite natural to wonder whether the results of the previous sections, in particular the numerical results of Section \ref{somenumres},
apply to different potentials $U$, other than \eq{quadratic}.\par
If we replace \eq{quadratic} with
$$U(y,z_1,z_2)=\frac{\gamma y^2 z_1^2+\beta y^2z_2^2+z_1^2z_2^2}{2}\quad \gamma, \beta >0\,,$$
then \eq{Mathieu system} becomes
$$
\begin{cases}
\ddot{\xi}_1+\left(\lambda_1^2+\frac{\gamma x_0^2}{2}+\frac{\gamma x_0^2}{2}\cos(2 \mu t)\right)\xi_1=0 &\\
\ddot{\xi}_2+\left(\lambda_2^2+\frac{\beta x_0^2}{2} +\frac{\beta x_0^2}{2} \cos(2 \mu t)\right)\xi_2=0\,. &\\
\end{cases}
$$
Whence, we still obtain Mathieu equations of the form \eq{Mathieu canon} but with
$$
\alpha_1=\frac{2\lambda_i^2+\gamma x_0^2}{2\mu^2} \quad \text{and} \quad q_1=\frac{\gamma x_0^2}{4\mu^2}\,, \quad \alpha_2=\frac{2\lambda_i^2+\beta x_0^2}{2\mu^2} \quad \text{and} \quad q_2=\frac{\beta x_0^2}{4\mu^2}\,.
$$
We note that in both the cases there holds $\alpha_i=\frac{\lambda_i^2}{\mu^2}+2q_i$. Hence, we have different parametrizations of the same parallel lines. In terms of our stability analysis the values of $\gamma$ and $\beta$ may be exploited to increase or decrease the energy threshold for the stability of the corresponding equations, see the proof of Theorem \ref{stability}. \par
More generally, let $U$ be a non-negative, differentiable function with locally Lipschitz derivatives such that $\nabla U(y,0,0)=(0,0,0)$ for all $y\in \R$. Then, all the above analysis holds and \eq{Mathieu system} becomes
\neweq{linsyst}
\begin{cases}
\ddot{\xi}_1+\left(\lambda_1^2+U_{z_1z_1}(x_0 \cos( \mu t),0,0)\right)\xi_1=0 &\\
\ddot{\xi}_2+\left(\lambda_2^2+U_{z_2z_2}(x_0 \cos( \mu t),0,0)\right)\xi_2=0. &\\
\end{cases}
\endeq
One may obtain different lines, other than \eq{retteparametriche}, for instance by taking non-polynomial potentials $U$, in which case Hill equations show up instead of the simpler Mathieu equations in \eq{Mathieu canon}. Then, the stability regions may have strange shapes (see  \cite{broer}) and it becomes more difficult to determine a precise criterion governing the energy transfer between modes.
\par\smallskip
Notice that if the potential $U=U(y,z_1,z_2)$ satisfies
\neweq{degenerate}
U_{z_1z_1}(y,0,0)=U_{z_2z_2}(y,0,0)=0\qquad\forall y\in\R\,,
\endeq
then the linearized problem \eq{linsyst} simply becomes
\neweq{constant}
\ddot{\xi}_1+\lambda_1^2\xi_1=0\ ,\quad \ddot{\xi}_2+\lambda_2^2\xi_2 =0
\endeq
and is therefore independent of $y$ and of its amplitude of oscillation. As an example, consider the potential
$$U(y,z_1,z_2)=\frac{y^4 z_1^4+y^4z_2^4+z_1^4z_2^4}{4}$$
so that \eq{PG} becomes
\neweq{PG4}
\left\{\begin{array}{lll}
\ddot{y}+\mu^2 y +(z_1^4+z_2^4)y^3=0   &\ y(0)=x_0, \ \dot{y}(0)=0\\
\ddot{z}_1+\lambda_1^2 z_1+(y^4+z_2^4)z_1^3=0 &\ z_1(0)=\eps x_0, \ \dot{z}_1(0)=0\\
\ddot{z}_2+\lambda_2^2 z_2+(y^4+z_1^4)z_2^3=0 &\ z_2(0)=\eps x_0, \ \dot{z}_2(0)=0\,.
\end{array}\right.
\endeq

In this case, the parametric equations \eq{retteparametriche} make no sense and the corresponding (green) lines in Figure~\ref{calcoloenergia} are
horizontal: this is why we call this case \textit{degenerate}. We have tried some numerical experiments; let us describe some of the results we obtained.\par

$\bullet$ If $\mu=\lambda_1=1$, $\lambda_2=2$ and $\eps=10^{-3}$, the system was extremely unstable. The residual mode $z_1$ started capturing the energy of $y$ even for small
values of $x_0$. With some fine experiments we could detect instability already for $x_0=0.5$, but we suspect the system to be unstable since the
very beginning. 
Completely similar results were obtained for other choices of $\lambda_2>\lambda_1=\mu$. And also the
case $\lambda_2=\lambda_1=\mu$ gave similar response with the addition (of course!) that both $z_1$ and $z_2$ captured the energy of $y$.\par
$\bullet$ If $\mu=1$, $\lambda_1=\sqrt2$, $\lambda_2=2$, $x_0=1$ and $\varepsilon=0.5$, a large $\varepsilon$ compared with Section \ref{somenumres}; the reason of this choice is that for smaller $\eps$ no interesting phenomenon was evident. We found that the dominating mode $y$ captured some small amount of energy from the residual mode $z_2$. Therefore, it is {\em not true} that the energy
always moves from the dominating to a residual mode, also the dominating mode can capture the energy and become ``more dominating''.
This seems to be related to the ``end of the black bumps'' displayed in many plots, see e.g.\ Figure \ref{typical}, namely to the interval of time where the residual mode returns the energy to the dominating mode.\par
$\bullet$ If $\mu=1$, $\lambda_1=\sqrt2$, $\lambda_2=2$ and $\eps=10^{-3}$, we could see some energy going from $y$ to $z_2$ only for $x_0\ge10$. Therefore,
the system turned out to be very stable. We suspect that, again, the ratios $\lambda_i/\mu$ play a major role.\par\medskip\par
What we have seen in this section suggests that degenerate problems such as \eq{PG4} are either extremely unstable (manifesting
instability for very small
energies) or extremely stable with instability appearing only for very large energies. This alternative depends on the ratios $\lambda_i/\mu$. It is
also clear that \eq{PG4} cannot remain stable for any energy since \eq{constant} fails to take into account both the interactions between the residual
modes and the perturbations of the periodic solution $y(t)=x_0\cos(\mu t)$: these are fairly small but for large energies they certainly play some role.\par
Summarizing, the degenerate problem \eq{PG4}, where  \eq{degenerate} holds, and the corresponding linearized problem \eq{constant} behave
quite differently when compared to \eq{hamilton2} and a neat pattern as the one described in Section \ref{somenumres} is not available.

\section{Mechanical interpretation and structural remedies}\label{mec}

In this section we aim to justify from a mechanical point of view the numerical results found in the previous sections. Let us first summarize the
main phenomena observed.\par
(I) As long as the two couples of parameters $(q,a)$ of \eq{retteparametriche} lie in the (white) stability region of the Mathieu diagram, see
Figures \ref{enlarged} and \ref{general}, the solution $(\bar y,0,0)=(x_0\cos(\mu t),0,0)$ to system \eq{hamilton2} is stable, see Definition \ref{def:stable}.\par
(II) When a couple $(q,a)$ lies in an instability region and is sufficiently far from the stability region, then the corresponding residual modes become
fairly large.\par
(III) When the couple $(q,a)$ lies in an instability region but is close to the stability region, our numerical results could not detect a
neat instability.\par\medskip
The most intriguing result is certainly (III). In order to better understand it, we compared this behavior with the somehow related behavior of the classical linear
Mathieu equation
\neweq{ancoraMathieu}
\ddot{w}+\Big(a+2q \cos(2 t)\Big) w=0\,.
\endeq
To obtain two independent solutions, we plotted the two solutions with initial data $(w(0),\dot{w}(0))\in\{(1,0);(0,1)\}$.
We analyzed in particular the two first instability regions. From \eq{qzero} we know that $(q,a)$ lies in the first (resp.\ second)
instability region for small enough $q$ if
$$1-q+O(q^2)<a<1+q+O(q^2)\qquad\mbox{\Big(resp. }4-\frac{1}{12}q^2+O(q^4)<a<4+\frac{5}{12}q^2+O(q^4)\, \mbox{\Big)}.$$
Therefore, we considered couples such as $(q,a)=(q,1)$ and $(q,a)=(q,4)$ for $q>0$ sufficiently small and we could observe the following facts.\par
(IV) The solutions were always unbounded (thereby confirming instability).\par
(V) For very small $q$ the solutions became large only after a long interval of time.\par
(VI) For larger values of $q$ the solutions became large much earlier in time.\par
(VII) For the same $q>0$ the instability was more evident when $a=1$ than when $a=4$.\par\medskip
The observation (VII) appears strictly related to (II) and (III) and enables us to conclude that
\begin{center}
\begin{minipage}{155mm}
{\em if the couple $(q,a)$ lies in the instability region of the Mathieu diagram, then the instability of the trivial solution of \eqref{ancoraMathieu}
increases with the distance of the couple $(q,a)$ from the stability regions.}
\end{minipage}
\end{center}

The model system \eq{hamilton2} is nonlinear and all its solutions are bounded in view of the energy conservation. Whence, we cannot expect that its solutions
start increasing in amplitude as for \eq{ancoraMathieu}. Roughly speaking,
\begin{center}
\begin{minipage}{155mm}
{\em when the residual mode exhibits a tendency to grow up, the energy conservation bounces it back and decreases its amplitude.}
\end{minipage}
\end{center}
We can however expect that the residual modes start growing up earlier in time and wider in amplitude if the parameters are far from the stability region.
This is precisely what we saw in our experiments, see Figure \ref{primaterna}. In particular, when the parametric lines \eq{rette} reach and intersect
a thin instability region (one of the cusps close to some $a=n^2$ with $n\ge2$), the parameters are so close to the stability region that the energy
inhibits the residual modes to capture a significant amount of energy. From the physical point of view, the instability which occurs when the lines
\eq{rette} cross a thin cusp is irrelevant, both because it has low probability to occur and because, even if it occurs, the residual
mode remains fairly small. In turn, from the mechanical point of view, we know that small torsional oscillations are harmless and the bridge would
remain safe. Summarizing, we conclude that
\begin{center}
\begin{minipage}{155mm}
{\em when the parametric lines \eqref{rette} cross a thin instability region, only small torsional oscillations appear and the bridge basically
remains stable.}
\end{minipage}
\end{center}

{From} the Mathieu diagram and from the asymptotic expansions of the characteristic curves, see \cite[Sections 2.151]{mc}, we learn that the instability regions become more narrow
as $a=n^2$ increases. Since the parametric lines \eq{rette} take their origin when $a=\lambda_i^2/\mu^2$ (see the right picture
in Figure \ref{calcoloenergia}), it would be desirable that $\lambda_i\gg\mu$. This gives a structural remedy to improve the torsional stability of
a bridge:
\begin{center}
\begin{minipage}{155mm}
{\em the torsional stability of a suspension bridge depends on the ratios between the torsional frequencies and the vertical frequencies; the larger
they are, more stable is the bridge.}
\end{minipage}
\end{center}
Therefore, our results suggest that bridges should be designed in such a way that these ratios are very large.
\par
\bigskip
\par

\textbf{Acknowledgments.} \noindent
The first and third Authors are partially supported by the Research Project FIR (Futuro in Ricerca) 2013 \emph{Geometrical and qualitative aspects of PDE's}.
The second Author is partially supported by the PRIN project {\em Equazioni alle derivate parziali di tipo ellittico e parabolico: aspetti geometrici,
disuguaglianze collegate, e applicazioni}. The three Authors are members of the Gruppo Nazionale per l'Analisi Matematica, la Probabilit\`a e le loro
Applicazioni (GNAMPA) of the Istituto Nazionale di Alta Matematica (INdAM).

\end{document}